\documentclass{amsart}

\usepackage{hyperref}

\usepackage{amssymb}
\usepackage{amsmath}
\usepackage{amsthm}
\usepackage{enumerate}
\usepackage{graphicx}

\usepackage{color}

\theoremstyle{plain}

\makeatletter
\newtheorem*{rep@theorem}{\rep@title}
\newcommand{\newreptheorem}[2]{%
\newenvironment{rep#1}[1]{%
 \def\rep@title{#2 \ref{##1}}%
 \begin{rep@theorem}}%
 {\end{rep@theorem}}}
\makeatother

\newtheorem{theorem}{Theorem}[section]
\newreptheorem{theorem}{Theorem}
\newtheorem{proposition}[theorem]{Proposition}
\newtheorem{lemma}[theorem]{Lemma}
\newtheorem{corollary}[theorem]{Corollary}

\theoremstyle{definition}
\newtheorem{definition}[theorem]{Definition}
\newtheorem{example}{Example}[section]

\theoremstyle{remark}
\newtheorem*{remark}{Remark}

\newtheorem*{acknowledgments}{Acknowledgments}

\newcommand{\RiemannSphere}{\widehat{\mathbb{C}}}

\renewcommand {\tilde} {\widetilde}

\begin{document}

\title{Fekete polynomials and shapes of Julia sets}

\date{May 19 2017}

\author[K. Lindsey]{Kathryn A. Lindsey}
\thanks{First author supported by an NSF Mathematical Sciences Research Postdoctoral Fellowship}
\address{Department of Mathematics, University of Chicago, Chicago,  IL 60637, United States.}
\email{klindsey@math.uchicago.edu}
\author[M. Younsi]{Malik Younsi}
\thanks{Second author supported by NSERC and NSF Grant DMS-1664807}
\address{Department of Mathematics, Stony Brook University, Stony Brook, NY 11794-3651, United States.}
\email{malik.younsi@gmail.com}


\begin{abstract}
We prove that a nonempty, proper subset $S$ of the complex plane can be approximated in a strong sense by polynomial filled Julia sets if and only if $S$ is bounded and $\hat{\mathbb{C}} \setminus \textrm{int}(S)$ is connected.  The proof that such a set is approximable by filled Julia sets is constructive and relies on Fekete polynomials. Illustrative examples are presented. We also prove an estimate for the rate of approximation in terms of geometric and potential theoretic quantities.
\end{abstract}

\maketitle

\section{Introduction}

The study of possible shapes of polynomial Julia sets was instigated by the first author in \cite{LIN2}, who proved that any Jordan curve in the complex plane $\mathbb{C}$ can be approximated by polynomial Julia sets arbitrarily well in the Hausdorff distance, thereby answering a question of W.P. Thurston.

\begin{theorem}[Lindsey \cite{LIN2}]
\label{thmKL}
Let $E \subset \mathbb{C}$ be any closed Jordan domain. Then for any $\epsilon>0$, there exists a polynomial $P$ such that
$$d(E,\mathcal{K}(P))<\epsilon, \quad d(\partial E, \mathcal{J}(P))<\epsilon.$$
\end{theorem}

\noindent Here $\mathcal{K}(P):=\{z \in \mathbb{C} : P^m(z) \nrightarrow \infty \,\, \mbox{as}\,\, m \to \infty\}$ is the filled Julia set of $P$, $\mathcal{J}(P):=\partial \mathcal{K}(P)$ is the Julia set of $P$ and $d$ is the Hausdorff distance.

Note that Theorem \ref{thmKL} remains valid if $E$ is any nonempty connected compact set with connected complement, by a simple approximation process.

Approximating compact sets by fractals has proven over the years to be a fruitful technique in the study of important problems in complex analysis, such as the universal dimension spectrum for harmonic measure (cf. the work of Carelson--Jones \cite{CAR2}, Binder--Makarov--Smirnov \cite{BIN}, and the references therein). Other related works include \cite{BIS}, where it was shown that any connected compact set in the plane can be approximated by dendrite Julia sets, and \cite{KIM}, containing applications to computer graphics.

An important feature of the approach in \cite{LIN2} which seems to be absent from other works is that it is constructive and can easily be implemented to obtain explicit images of Julia sets representing various shapes.

In this paper, we prove the following generalization of Theorem \ref{thmKL}.

\begin{theorem}
\label{mainthm1}
Let $E \subset \mathbb{C}$ be any nonempty compact set with connected complement. Then for any $\epsilon>0$, there exists a polynomial $P$ such that
$$d(E,\mathcal{K}(P))<\epsilon, \quad d(\partial E, \mathcal{J}(P))<\epsilon.$$
\end{theorem}
Note that here $E$ is not assumed to be connected.

This supersedes \cite[Theorem 4.3]{LIN2}, which deals with rational maps instead of polynomials.

The proof of Theorem \ref{mainthm1} generalizes the method introduced in \cite{LIN2} by putting it in the more natural framework of potential theory. More precisely, our approach is based on the observation that in Theorem \ref{thmKL}, the roots of the polynomial $P$ are equidistributed with respect to harmonic measure on $\partial E$. In the connected case, such points can be obtained as images of equally spaced points on the unit circle under the Riemann map. This, however, does not extend to disconnected sets and in this case, the explicit construction of equidistributed sequences of points is more difficult. Nevertheless, this is a classical problem in potential theory which has been extensively studied in the past. For instance, Fekete described in \cite{FEK} the following construction :

Let $E$ be a nonempty compact set with connected complement, and let $n \geq 2$. A \textit{Fekete} $n$-\textit{tuple} for $E$ is any $n$-tuple $w_1^n, \dots, w_n^n \in E$ which maximizes the product
$$\prod_{j<k}|w_j^n-w_k^n|.$$
Note that by the maximum principle, any Fekete $n$-tuple lies in $\partial E$.

Fekete points are classical objects of potential theory which have proven over the years to be of fundamental importance to a variety of problems related to polynomial interpolation. For instance, they can be used to give a proof of Hilbert's Lemniscate Theorem, which states that any compact plane set with connected complement can be approximated by polynomial lemniscates. In fact, as we will see, Hilbert's theorem is closely related to Theorem \ref{mainthm1}. For more information on Fekete points and their applications, we refer the reader to \cite{BOS} and \cite[Chapter III, Section 1]{SAF}.

It is not difficult to show that Fekete points are indeed equidistributed with respect to harmonic measure on $\partial E$, in the sense that the counting measures
$$\mu_n:= \frac{1}{n} \sum_{j=1}^n \delta_{w_j^n}$$
converge $\operatorname{weak}^*$ to the harmonic measure (see Proposition \ref{weakconvergence}). Combining this with the method put forward in \cite{LIN2} gives a constructive proof of Theorem \ref{mainthm1}.

Other equidistributed sequences of points include Leja points, which share many nice properties with Fekete points but are easier to compute numerically. Theorem \ref{thmExplicitLeja} proves that our approximation technique works if Leja points, instead of Fekete points, are used. The technique in \cite{LIN2} for constructing polynomials whose filled Julia sets approximate a given Jordan domain required first obtaining an approximation of the Riemann map on the complement of that domain.  Using Leja points or Fekete points obviates this step, and thus constitutes an advance in the constructive process.

Note that Theorem \ref{mainthm1} not only yields approximation of the set $E$ by the filled Julia set of a polynomial, but also approximation of its boundary by the corresponding Julia set. Consequently, we introduce the following definition.

\begin{definition}
\label{DefTotApprox}
A nonempty proper subset $S$ of the plane is \textit{totally approximable} by a collection $\mathcal{C}$ of nonempty proper subsets of the plane if for any $\epsilon>0$, there exists $C \in \mathcal{C}$ such that
$$d(S,C)<\epsilon, \quad d(\partial S, \partial C)<\epsilon.$$
\end{definition}

Note that closeness of two sets in the Hausdorff distance neither implies nor is implied by closeness of their boundaries, see Figure \ref{f:HausdorffDistance}. On the other hand, it is easy to see that closeness of both the sets and their boundaries implies closeness of the complements, c.f. Lemma \ref{lem2}.

 \begin{figure}[!h] \label{f:HausdorffDistance}
  \centering
  \includegraphics[width=.6\linewidth]{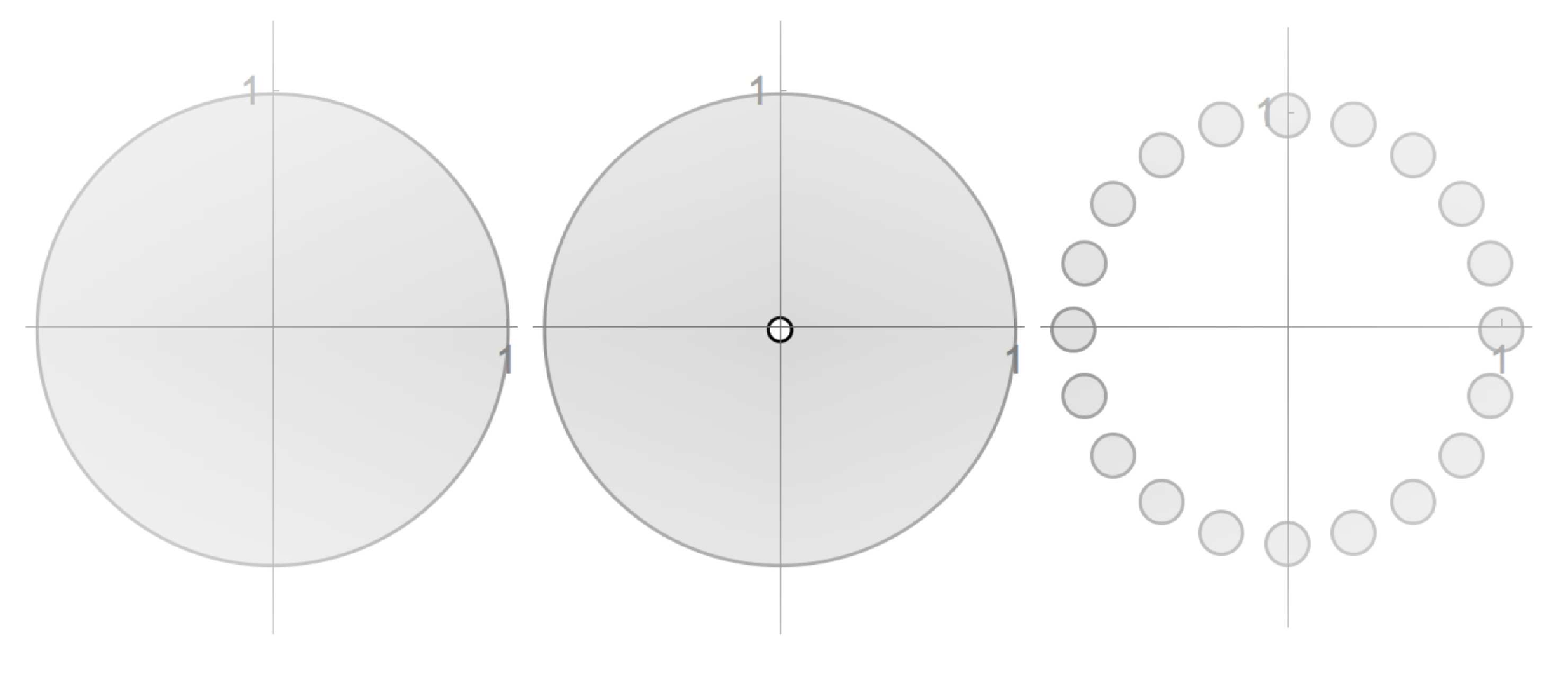}
  \caption[]{A unit disk, a unit disk with a small disk removed, and a ring of small disks.  The first two are close in the Hausdorff distance, but their boundaries are not.  The boundaries of the first and third are close in the Hausdorff distance, but the sets themselves are not. }
  \end{figure}

With this definition, Theorem \ref{mainthm1} states that any nonempty compact plane set $E$ with connected complement is totally approximable by polynomial filled Julia sets. More generally, using an approximation process, one can prove that this remains true if $E$ is replaced by any nonempty bounded set whose interior has connected complement. This condition also turns out to be necessary.

\begin{theorem}
\label{mainthm3}
A non-empty proper subset $S$ of the complex plane $\mathbb{C}$ is totally approximable by polynomial filled Julia sets if and only if $S$ is bounded and $\mathbb{C} \setminus \operatorname{int}(S)$ is connected.
\end{theorem}

Our method also gives a precise estimate for the rate of approximation in Theorem \ref{mainthm1}. More precisely, let $E \subset \mathbb{C}$ be a compact set with connected complement $\Omega$ in the Riemann sphere $\RiemannSphere$. We also assume that the interior of $E$ is not empty. In particular, the Green's function $g_{\Omega}(\cdot,\infty)$ for $\Omega$ with pole at $\infty$ exists. For $n \in \mathbb{N}$, define $s_n(E)$ to be the infimum of $s>0$ for which there exists a polynomial $p_n$ of degree $n$ such that
$$E \subset \mathcal{K}(p_n) \subset E_s,$$
where $E_s:= E \cup \{z \in \Omega : g_\Omega(z,\infty) \leq s\}$.

\begin{theorem}
\label{mainthm2}
Let $E \subset \mathbb{C}$ be a uniformly perfect compact set with nonempty interior and connected complement. Then there exists a real number $c=c(E)$ depending only on $E$ such that
$$s_n(E) \leq c \, \frac{\log{n}}{\sqrt{n}} \qquad (n \geq 1).$$
\end{theorem}

Recall that a compact subset $E$ of the plane is \textit{uniformly perfect} if there exists a real number $a>0$ such that for any $z \in E$ and for any $0<r<\operatorname{diam}(E)$, there is a point $w \in E$ with $ar \leq |z-w| \leq r$. This condition on $E$ is not too restrictive, since, for example, uniformly perfect sets include compact sets consisting of finitely many nontrivial connected components.

A direct consequence of Theorem \ref{mainthm2} is that $s_n(E) \to 0$ as $n \to \infty$. In particular, Theorem \ref{mainthm2} implies Theorem \ref{mainthm1} for uniformly perfect sets, since the compact sets $E_s$ shrink to $E$ as $s$ decreases to $0$. We also mention that results similar to Theorem \ref{mainthm2} were obtained in \cite{AND} and \cite{LAI}, but for the rate of approximation by polynomial lemniscates in Hilbert's Lemniscate Theorem.

Finally, we address the natural question of whether a nonempty connected compact set with connected complement can be approximated by Jordan domain Julia sets.  We prove that, in this case, the Julia sets of our approximating polynomials are not only Jordan curves but quasicircles.

\begin{theorem}\label{t:connected}
Let $E$ be a nonempty connected compact set with connected complement. Then there are polynomials $(P_n)_{n \geq 1}$ whose filled Julia sets are closed Jordan domains totally approximating $E$ in the sense of Definition \ref{DefTotApprox}. Moreover, the polynomials can be constructed such that each $P_n$ is a hyperbolic polynomial whose Julia set $\mathcal{J}(P_n)$ is a quasicircle.
\end{theorem}

The remainder of the paper is organized as follows. Section \ref{sec2} contains various preliminaries from potential theory. In Section \ref{sec3}, we prove an explicit version of Theorem \ref{mainthm1} using polynomials with zeros at points of any equidistributed sequence with respect to harmonic measure. Section \ref{sec4} is devoted to the special case of Fekete polynomials, which yields a proof of Theorem \ref{mainthm2}. Then, in Section \ref{sec5}, we give an alternate proof of Theorem \ref{mainthm1} using Hilbert's Lemniscate Theorem. Section \ref{sec6} contains the proof of Theorem \ref{mainthm3}. In Section \ref{sec7}, we prove Theorem \ref{t:connected}. Lastly, Section \ref{sec8} is devoted to numerical examples. We first discuss Leja points and prove that our approximation scheme remains valid if Fekete points are replaced by Leja points. This allows us to compute images of Julia sets representing various disconnected shapes.

\section{Preliminaries from potential theory}
\label{sec2}
This section contains various preliminaries from potential theory, including the notions of potential, logarithmic capacity and Green's function. The proofs are quite standard and are found in \cite{RAN} for example, but we include them for the reader's convenience.

We begin with the notions of potential and energy of a measure.

\begin{definition}
Let $\mu$ be a finite positive Borel measure on $\mathbb{C}$ with compact support.

\begin{enumerate}[\rm(i)]
\item The \textit{potential} of $\mu$ is the function $p_\mu : \mathbb{C} \to [-\infty,\infty)$ defined by
$$p_\mu(z):=\int \log{|z-w|} \, d\mu(w) \qquad (z \in \mathbb{C}).$$

\item The \textit{energy} of $\mu$, noted $I(\mu)$, is given by
$$I(\mu):= \int \int \log{|z-w|} \, d\mu(z) \, d\mu(w) = \int p_\mu(z) \, d\mu(z).$$
\end{enumerate}
\end{definition}
One can check that $p_\mu$ is subharmonic on $\mathbb{C}$, harmonic on $\mathbb{C} \setminus (\operatorname{supp}\mu)$ and that
\begin{equation}
\label{eqpotential}
p_{\mu}(z) = \mu(\mathbb{C}) \log{|z|} + O(|z|^{-1})
\end{equation}
as $z \to \infty$.

Now, let $E$ be a nonempty compact subset of $\mathbb{C}$. We assume that the complement of $E$ is connected. Let $\mathcal{P}(E)$ denote the collection of all Borel probability measures supported on $E$.
\begin{definition}
A measure $\nu \in \mathcal{P}(E)$ is an \textit{equilibrium measure} for $E$ if
$$I(\nu) = \sup_{\mu \in \mathcal{P}(E)} I(\mu).$$
\end{definition}

A standard $\operatorname{weak}^*$-convergence argument shows that $E$ always has an equilibrium measure, see \cite[Theorem 3.3.2]{RAN}. Moreover, if the supremum in the definition is not $-\infty$, then the equilibrium measure $\nu$ for $E$ is always unique and supported on $\partial E$ (\cite[Theorem 3.7.6]{RAN}). It also coincides with $\omega_{\Omega}(\cdot, \infty)$, the harmonic measure for $\Omega:=\RiemannSphere \setminus E$ and $\infty$.

The energy of the equilibrium measure is used to define the logarithmic capacity of $E$.

\begin{definition}
The \textit{logarithmic capacity} of $E$ is defined by
$$\operatorname{cap}(E):=e^{I(\nu)},$$
where $\nu$ is the equilibrium measure for $E$.
\end{definition}

\begin{remark}
The quantity $-I(\nu)=-\log{\operatorname{cap(E)}}$ is usually called \textit{Robin's constant}.
\end{remark}

\noindent For example, the capacity of a closed disk is equal to its radius and the capacity of a segment is equal to a quarter of its length.

For the rest of the paper, we shall suppose that $\operatorname{cap}(E)>0$. This can always be assumed without loss of generality in Theorem \ref{mainthm1}, adjoining a small disk in $E$ if necessary. The condition of positive logarithmic capacity is enough to ensure that the Green's function exists.

\begin{definition}
Let $\Omega:=\RiemannSphere \setminus E$ be the complement of $E$ in the Riemann sphere. The \textit{Green's function} for $\Omega$ with pole at $\infty$ is the unique function $g_{\Omega}(\cdot,\infty):\Omega \to (0,\infty]$ such that
\begin{enumerate}[\rm(i)]
\item $g_{\Omega}(\cdot,\infty)$ is harmonic on $\Omega \setminus \{\infty\}$;
\item $g_{\Omega}(\infty,\infty)=\infty$ and as $z \to \infty$,
$$g_{\Omega}(z,\infty) = \log{|z|} + O(1);$$
\item $g_{\Omega}(z,\infty) \to 0$ as $z \to \zeta$, for all $\zeta \in \partial \Omega$ except possibly a set of zero logarithmic capacity.
\end{enumerate}
\end{definition}

The Green's function for $\Omega$ with pole at $\infty$ can be recovered from the potential of the equilibrium measure $\nu$.

\begin{theorem}[Frostman's theorem]
\label{TheoFrostman}
If $\nu$ is the equilibrium measure for $E$, then
\begin{enumerate}[\rm(i)]
\item $p_\nu \geq I(\nu)$ on $\mathbb{C}$;
\item $p_\nu(z) = I(\nu)$ for all $z \in E$ except possibly a set of zero logarithmic capacity;
\item $g_{\Omega}(z,\infty) = p_{\nu}(z)-I(\nu) \qquad (z \in \Omega \setminus \{\infty\}).$
\end{enumerate}
\end{theorem}
See (\cite[Theorem 3.3.4]{RAN}). Note that combining (iii) with Equation (\ref{eqpotential}), we get
$$I(\nu)=\log \operatorname{cap}(E) = \lim_{z \to \infty} (\log{|z|} - g_{\Omega}(z,\infty)).$$

\begin{remark}
For sufficiently nice compact sets $E$, the potential $p_\nu$ is continuous on $\mathbb{C}$ and the exceptional set of zero capacity in (ii) is in fact empty. More precisely, again under the assumption that $\operatorname{cap}(E)>0$, the equality $p_\nu(z) = I(\nu)$ holds for all $z \in E$ if and only if every point of $\partial E$ is regular for the Dirichlet Problem. A sufficient condition for this to hold is that $E$ is uniformly perfect. See \cite[Theorem 3.1.3]{RAN} and \cite[Theorem 4.2.4]{RAN}
\end{remark}

\section{Constructive approximation by Julia sets}
\label{sec3}

In this section, we prove an explicit version of Theorem \ref{mainthm1}.

As before, let $E \subset \mathbb{C}$ be a compact set with connected complement $\Omega$. We will construct polynomials whose filled Julia sets totally approximate $E$, in the sense of Definition \ref{DefTotApprox}.

First, we can assume that the interior of $E$ is not empty, adjoining a small disk in $E$ if necessary. We can further suppose without loss of generality that $0$ is an interior point of $E$, since otherwise it suffices to conjugate the polynomial by a suitable translation. Note that the above ensures that $\operatorname{cap}(E)>0$. For reasons that will soon become clear, we shall assume in addition that $E$ is uniformly perfect.

Now, let $(q_n)_{n \geq 1}$ be any sequence of monic polynomials of degree $n$ having all their zeros in $\partial E$ such that the counting measures
$$\mu_n:= \frac{1}{n} \sum_{\zeta \in q_n^{-1}(\{0\})} \delta_{\zeta}$$
converge $\operatorname{weak}^*$ to $\nu$, the equilibrium measure for $E$.

We will need the following lemma on the behavior of $|q_n|^{1/n}$.

\begin{lemma}
\label{lemmaTom}
We have

\begin{enumerate}[\rm(i)]
\item $|q_n|^{1/n} \to e^{p_{\nu}}$ locally uniformly on $\Omega = \RiemannSphere \setminus E$;
\item $\lim_{n \to \infty} \|q_n\|_E^{1/n} = \operatorname{cap}(E)$.
\end{enumerate}
Here $\|q_n\|_E =  \sup_{z \in E} |q_n(z)|.$

\end{lemma}

\begin{proof}
The following proof is due to Thomas Ransford.

To prove (i), fix $z \in \mathbb{C} \setminus E$. Since the function $w \mapsto \log{|z-w|}$ is continuous on $\partial E$, we have, by $\operatorname{weak}^*$ convergence,
$$\int_{\partial E} \log{|z-w|} \, d\mu_n(w) \to \int_{\partial E} \log{|z-w|} \, d\nu(w),$$
i.e.
\begin{equation}
\label{simpleconv}
\frac{1}{n} \log{|q_n(z)|} \to p_\nu(z),
\end{equation}
and this holds for all $z \in \mathbb{C} \setminus E$. Now, let $D$ be a disk whose closure is contained in $\mathbb{C} \setminus E$. Note that the sequence $(q_n^{1/n})_{n \geq 1}$ is uniformly bounded on $D$, so by Montel's theorem and (\ref{simpleconv}), every subsequence has a subsequence converging uniformly to some analytic function $h$ on $D$ with $|h|=e^{p_{\nu}}$. This shows that $|q_n|^{1/n} \to e^{p_{\nu}}$ locally uniformly on $\mathbb{C} \setminus E$. The fact that $|q_n|^{1/n} \to e^{p_{\nu}}$ uniformly near $\infty$ follows from Equation (\ref{eqpotential}).

To prove (ii), first note that
$$\int_{\partial E} \frac{1}{n} \log{|q_n(z)|} \, d\nu(z) = \int_{\partial E} \int_{\partial E} \log{|z-w|} \, d\nu(z) d\mu_n(w) = \int_{\partial E} p_{\nu}(w) \, d\mu_n(w) \geq I(\nu),$$
the last inequality because $p_{\nu} \geq I(\nu)$ on $\mathbb{C}$, by Theorem \ref{TheoFrostman}. It follows that
\begin{equation}
\label{equa1}
\|q_n\|_E ^{1/n} \geq e^{I(\nu)}=\operatorname{cap}(E) \qquad (n \geq 1).
\end{equation}
For the other direction, recall that since $E$ is uniformly perfect, the potential $p_{\nu}$ is continuous on $\mathbb{C}$ and satisfies $p_{\nu}=I(\nu)$ everywhere on $E$ (see the remark following Theorem \ref{TheoFrostman}). Thus, given $\epsilon>0$, there exists a bounded neighborhood $V$ of $E$ such that $p_{\nu} \leq I(\nu) + \epsilon$ on $V$. By (i), the functions $|q_n|^{1/n}$ converge uniformly to $e^{p_\nu}$ on $\partial V$, so there exists $N$ such that, for all $n \geq N$ and $z \in \partial V$, we have
$$|q_n(z)|^{1/n} \leq e^{p_\nu(z)}+\epsilon \leq e^{\epsilon}\operatorname{cap}(E) + \epsilon.$$
By the maximum principle, the same inequality holds for all $n \geq N$ and all $z \in E$, from which it follows that
$$\limsup_{n \to \infty} \|q_n\|_E^{1/n} \leq e^{\epsilon}\operatorname{cap}(E) + \epsilon.$$
Letting $\epsilon \to 0$ and combining with (\ref{equa1}), we get (ii).

\end{proof}

We can now prove the following explicit version of Theorem \ref{mainthm1}.

\begin{theorem}
\label{thmExplicit}
Let $E$ and $(q_n)_{n \geq 1}$ be as above. For $s>0$ and $n \in \mathbb{N}$, define the polynomial
$$P_{n,s}(z):=z\frac{e^{-ns/2}}{\operatorname{cap}(E)^n} q_n(z).$$
Then for any bounded neighborhood $U$ of $E$, there exist $s$ and $n$ such that
$$E \subset \operatorname{int}(\mathcal{K}(P_{n,s})) \subset U.$$

\end{theorem}

In particular, the set $E$ is totally approximable by the filled Julia sets of the polynomials $P_{n,s}$, since $U$ can be made arbitrarily close to $E$.

\begin{proof}

Fix $s>0$ sufficiently small so that
$$g_{\Omega}(z,\infty)>s \qquad (z \notin U).$$
By Lemma \ref{lemmaTom}, we can choose $n \in \mathbb{N}$ sufficiently large so that the following conditions hold :

\begin{enumerate}[\rm(i)]
\item $\displaystyle \left(\frac{R}{r}\right)^{1/n} \frac{\|q_n\|_E^{1/n}}{\operatorname{cap}(E)} < e^{s/2} \qquad (z \in E)$
\item $\displaystyle \left| \frac{1}{n} \log{|q_n(z)|} - p_\nu(z) \right| \leq \frac{s}{4} \qquad (z \notin U)$
\item $\displaystyle r e^{ns/4} > R,$

\end{enumerate}
where $r, R>0$ are such that $\overline{\mathbb{D}}(0,r) \subset E$ and $U \subset \mathbb{D}(0,R)$.

First, note that condition (i) implies that for $z \in E$,
$$|P_{n,s}(z)| < R \frac{e^{-ns/2}}{\operatorname{cap}(E)^n} \|q_n\|_E \leq R e^{-ns/2} e^{ns/2}\frac{r}{R} = r,$$
so that in particular, $P_{n,s}(E) \subset \mathbb{D}(0,r) \subset E$, which clearly implies that $E \subset \operatorname{int}(\mathcal{K}(P_{n,s}))$.

On the other hand, if $z \notin U$, then by (ii) and Theorem \ref{TheoFrostman}, we get
$$g_\Omega(z,\infty) -\frac{s}{4} \leq \frac{1}{n} \log{|q_n(z)|} - I(\nu)$$
and thus
$$\exp \left(n \left(g_\Omega(z,\infty) - \frac{s}{2} - \frac{s}{4} \right) \right) \leq \frac{e^{-ns/2}}{\operatorname{cap}(E)^n} |q_n(z)| = \left|\frac{P_{n,s}(z)}{z}\right|.$$
Since $g_\Omega(z,\infty) > s$, we obtain
\begin{equation}
\label{eqb}
e^{ns/4}|z| < |P_{n,s}(z)|.
\end{equation}
Now, by (iii), we get
$$|P_{n,s}(z)| > re^{ns/4} > R,$$
so that $P_{n,s}(z) \notin U$ whenever $z \notin U$. Equation (\ref{eqb}) then shows that the iterates of $P_{n,s}$ converge to $\infty$ on $\RiemannSphere \setminus U$, so that $\RiemannSphere \setminus U \subset \RiemannSphere \setminus \mathcal{K}(P_{n,s})$ or, equivalently, $ \mathcal{K}(P_{n,s}) \subset U$.

\end{proof}

\section{Fekete points}
\label{sec4}

In this section, we will see how using Fekete points for the zeros of the polynomials $q_n$ of Theorem \ref{thmExplicit} gives a more explicit version of that result which has Theorem \ref{mainthm2} as a corollary.

First, we need some definitions and the basic properties of Fekete points. As before, let $E$ be a nonempty compact set with connected complement, and assume that $\operatorname{cap}(E)>0$.

\begin{definition}
For $n \geq 2$, the $n$\textit{-th diameter} of $E$ is given by
$$\delta_n(E):=\sup_{w_1,\dots,w_n\in E} \prod_{j<k}|w_j-w_k|^{\frac{2}{n(n-1)}}.$$

\end{definition}
Recall from the introduction that an $n$-tuple $w_1^n,\dots,w_n^n \in E$ for which the supremum is attained is called a Fekete $n$-tuple for $E$.

Since $E$ is compact and nonempty, a Fekete $n$-tuple always exists, though it need not be unique. Moreover, the maximum principle shows that any Fekete $n$-tuple lies in $\partial E$. Also, the quantity $\delta_2(E)$ is the usual diameter of $E$, and $\delta_n(E) \leq \delta_2(E)$ for all $n$. In fact, the sequence $(\delta_n(E))_{n \geq 2}$ is decreasing and has a limit, which is nothing other than the logarithmic capacity of $E$.

\begin{theorem}[Fekete--Szeg\"{o}]
\label{thmFekSze}
The sequence $(\delta_n(E))_{n \geq 2}$ is decreasing and
$$\lim_{n \to \infty} \delta_n(E) = \operatorname{cap}(E).$$
\end{theorem}
See \cite[Theorem 5.5.2]{RAN}.

\begin{definition}
A \textit{Fekete polynomial} for $E$ of degree $n$ is a polynomial of the form
$$q(z):=\prod_{j=1}^n(z-w_j^n),$$
where $w_1^n,\dots,w_n^n$ is a Fekete $n$-tuple for $E$.
\end{definition}

As before, we denote the sup-norm of the polynomial $q$ on $E$ by $\|q\|_E$, i.e.
$$\|q\|_E= \sup_{z \in E} |q(z)|.$$

\begin{theorem}
\label{thmFek}
Let $q$ be a monic polynomial of degree $n$. Then
$$\|q\|_E^{1/n} \geq \operatorname{cap}(E).$$
If in addition $q$ is a Fekete polynomial for $E$, then
$$\|q\|_E^{1/n} \leq \delta_n(E).$$
\end{theorem}

\begin{proof}

To prove the first statement, it suffices to observe that since $E \subset q^{-1}\left(\overline{\mathbb{D}}(0,\|q\|_E)\right)$, we have
$$\operatorname{cap}(E) \leq \operatorname{cap}\left(q^{-1}\left(\overline{\mathbb{D}}(0,\|q\|_E)\right) \right),$$
by the monotonicity of logarithmic capacity. The quantity on the right-hand side is easily seen to be equal to $\operatorname{cap}\left(\overline{\mathbb{D}}(0,\|q\|_E)\right)^{1/n} = \|q\|_E^{1/n}$, since in this case the Green's function can be identified explicitly. This proves the first statement.

To prove the second statement, write $q(z):=\prod_{j=1}^n(z-w_j^n)$ where $w_1^n,\dots,w_n^n$ is a Fekete $n$-tuple for $E$. If $z \in E$, then $z,w_1^n,\dots,w_n^n$ is an $(n+1)$-tuple in $E$, so
$$\prod_{i=1}^n |z-w_i^n| \prod_{j<k} |w_j^n-w_k^n| \leq \delta_{n+1}(E)^{n(n+1)/2},$$
and hence
$$|q(z)| \leq \frac{\delta_{n+1}(E)^{n(n+1)/2}}{\delta_n(E)^{n(n-1)/2}} \leq \frac{\delta_n(E)^{n(n+1)/2}}{\delta_n(E)^{n(n-1)/2}} = \delta_n(E)^n,$$
as required.

\end{proof}

Before stating the next result, we need the notion of Harnack distance.

\begin{definition}
Let $D$ be a domain in $\RiemannSphere$. Given $z,w \in D$, the \textit{Harnack distance} between $z$ and $w$ is the smallest number $\tau_D(z,w)$ such that, for every positive harmonic function $u$ on $D$,
$$\tau_D(z,w)^{-1} u(w) \leq u(z) \leq \tau_D(z,w)u(w).$$
\end{definition}
 The existence of $\tau_D$ is a simple consequence of Harnack's inequality. It is not difficult to show that $\log \tau_D(z,w)$ is a continuous semimetric on $D$ (\cite[Theorem 1.3.8]{RAN}).

\begin{theorem}[Bernstein's Lemma]
\label{thmBernstein}
Let $E$ be as before, and let $\Omega=\RiemannSphere \setminus E$. If $q$ is a polynomial of degree $n$, then
$$\left( \frac{|q(z)|}{\|q\|_E}\right)^{1/n} \leq e^{g_\Omega(z,\infty)} \qquad (z \in \Omega).$$
If in addition $q$ is a Fekete polynomial for $E$, then
$$\left( \frac{|q(z)|}{\|q\|_E}\right)^{1/n} \geq e^{g_\Omega(z,\infty)} \left( \frac{\operatorname{cap}(E)}{\delta_n(E)} \right)^{\tau_\Omega(z,\infty)} \qquad (z \in \Omega).$$

\end{theorem}

\begin{proof}
To prove the first statement, assume without loss of generality that $q$ is monic, and define
$$u(z):= \frac{1}{n} \log{|q(z)|} - \frac{1}{n} \log{\|q\|_E} - g_\Omega(z,\infty) \qquad (z \in \Omega \setminus \{\infty\}).$$
Then $u$ is subharmonic on $\Omega \setminus \{\infty\}$. Also,
$$u(z)=\log{|z|} - \frac{1}{n} \log{\|q\|_E} - \log{|z|} + \log{\operatorname{cap}(E)} + o(1)$$
as $z \to \infty$, and therefore setting
$$u(\infty) := \log{\operatorname{cap}(E)} - \frac{1}{n} \log{\|q\|_E}$$
makes $u$ subharmonic on $\Omega$. Now, since $\partial \Omega \subset E$, we have
$$\limsup_{z \to \zeta} u(z) \leq 0$$
for all $\zeta \in \partial \Omega$, and thus $u \leq 0$ on $\Omega$ by the maximum principle for subharmonic functions. Note that $u(\infty) \leq 0$ by Theorem \ref{thmFek}. This implies the result.

To prove the second statement, note that if $q$ is a Fekete polynomial for $E$, then in particular all its zeros lie in $E$, so that $u$ is actually harmonic on $\Omega$. Also, from the first part of the proof, we have $u \leq 0$ on $\Omega$. Therefore, applying the definition of Harnack distance to $-u$, we get
$$u(z) \geq \tau_\Omega(z,\infty)u(\infty) \qquad (z \in \Omega).$$
Now, by Theorem \ref{thmFek},
$$u(\infty) = \log{\operatorname{cap}(E)} - \frac{1}{n} \log{\|q\|_E} \geq \log{\operatorname{cap}(E)} - \log{\delta_n(E)}.$$
Combining the last two inequalities yields the desired conclusion.

\end{proof}

\begin{corollary}
\label{CorBern}
For $z \in \Omega$, we have
$$|z| \leq R(E) e^{g_\Omega(z,\infty)},$$
where $R(E):= \sup_{w \in E} |w|$.
\end{corollary}

\begin{proof}
This follows from the first part of Theorem \ref{thmBernstein} applied to $q(z):=z^n$.

\end{proof}

The following result shows that Fekete points are equidistributed with respect to the harmonic measure and therefore can be used as zeros of the approximating polynomials $q_n$ of Theorem \ref{thmExplicit}.

\begin{proposition}
\label{weakconvergence}

Let $E$ be as before and for each $n \geq 2$, let
$$\mu_n:= \frac{1}{n} \sum_{j=1}^n \delta_{w_j^n},$$
where $w_1^n, w_2^n, \dots, w_n^n$ is a Fekete $n$-tuple for $E$ and $\delta_{w_j^n}$ is the unit point mass at the point $w_j^n$. Then $\mu_n \to \nu$ $\operatorname{weak}^*$, where $\nu$ is the equilibrium measure for $E$.
\end{proposition}

\begin{proof}
First, recall that $\nu$ and $\mu_n$ for $n \geq 2$ are all supported on $\partial E$.

Now, let $\mu$ be any $\operatorname{weak}^*$-limit of the sequence of measures $(\mu_n)$. Then by the monotone convergence theorem, we have
$$I(\mu) = \lim_{m \to \infty} \int_{\partial E} \int_{\partial E} \max(\log{|z-w|},-m) d\mu(z) d\mu(w).$$
Now, by the Stone-Weierstrass theorem, we have that $\mu_n \times \mu_n \to \mu \times \mu$ $\operatorname{weak}^*$ on $\partial E \times \partial E$, so the above equality becomes
\begin{eqnarray*}
I(\mu) &=& \lim_{m \to \infty} \lim_{n \to \infty} \int_{\partial E} \int_{\partial E} \max(\log{|z-w|},-m) d\mu_n(z) d\mu_n(w)\\
&=& \lim_{m \to \infty} \lim_{n \to \infty} \frac{1}{n^2} \sum_{j=1}^n \sum_{k=1}^n \max(\log{|w_j^n-w_k^n|},-m)\\
&\geq& \lim_{m \to \infty} \lim_{n \to \infty} \left(\frac{-m}{n} + \frac{2}{n^2} \sum_{j < k} \log{|w_j^n-w_k^n|}\right)\\
&=& \lim_{n \to \infty} \frac{n-1}{n} \log{\delta_n(E)}.
\end{eqnarray*}
This last limit is equal to $\log \operatorname{cap}(E) = I(\nu)$, by Theorem \ref{thmFekSze}. The maximality and uniqueness of the equilibrium measure $\nu$ then implies that $\mu=\nu$, as required.

\end{proof}

Combining this with Lemma \ref{lemmaTom}, we get that if $w_1^n, w_2^n, \dots, w_n^n$ is a Fekete $n$-tuple for $E$, then
$$\frac{1}{n} \sum_{j=1}^n \log{|z-w_j^n|} \to p_\nu(z)$$
as $n \to \infty$ for $z \in \Omega$, and the convergence is uniform on compact subsets. The following result due to Pritsker yields a precise rate of convergence, under the additional assumption that $E$ is uniformly perfect.

\begin{theorem}[Pritsker]
\label{thmPRI}
Let $E$, $w_1^n, w_2^n,\dots,w_n^n$, $n \geq 2$, be as above, and assume in addition that $E$ is uniformly perfect. Then
\begin{equation}
\label{eqprit1}
\left| \frac{1}{n} \sum_{j=1}^n \log{|z-w_j^n|} - p_\nu(z) \right| \leq C \frac{\log n}{\sqrt{n}} \qquad \left(z \in \Omega, \, g_\Omega(z,\infty)>\frac{1}{n}\right)
\end{equation}
where $C=C(E)$ is a constant depending only on $E$.

\end{theorem}

\begin{definition}
\label{defPRI}
We define the \textit{Pritsker constant} of $E$ to be the smallest constant $C(E)$ for which (\ref{eqprit1}) holds.
\end{definition}

We shall also need the following result, also due to Pritsker, on the rate of convergence of the $n$-th diameter to the logarithmic capacity.

\begin{theorem}[Pritsker]
\label{thmPRI2}
Under the assumptions of Theorem \ref{thmPRI}, we have
$$\log{\frac{\delta_n(E)}{\operatorname{cap}(E)}} \leq C_2 \frac{\log{n}}{\sqrt{n}},$$
where $C_2=C_2(E)$ is a constant depending only on $E$.

\end{theorem}

See \cite[Theorem 2.2]{PRI} and its corollary.

We now have everything that we need in order to prove a more explicit version of Theorem \ref{thmExplicit} based on Fekete points. As a consequence, we shall obtain a proof of Theorem \ref{mainthm2} on the rate of approximation by Julia sets.

As before, let $E \subset \mathbb{C}$ be a compact set with connected complement $\Omega$. As in Section \ref{sec3}, assume without loss of generality that $0$ is an interior point of $E$, so that in particular $\operatorname{cap}(E)>0$ and the Green's function $g_\Omega(\cdot,\infty)$ is well-defined.

We now introduce two positive geometric quantities associated with the set $E$, which we call the \textit{inner radius} and \textit{outer radius} of $E$. These are defined by
$$r(E):= \operatorname{dist}(0,\partial E)$$
and
$$R(E):= \sup_{w \in E} |w|$$
respectively. Note that $\overline{\mathbb{D}}(0,r(E)) \subset E \subset \overline{\mathbb{D}}(0,R(E)) $.

Finally, let us assume as well that $E$ is uniformly perfect. Recall that in this case, the Pritsker constant of $E$ is defined to be the smallest constant $C(E)$ such that

$$\left| \frac{1}{n} \sum_{j=1}^n \log{|z-w_j^n|} - p_\nu(z) \right| \leq C(E) \frac{\log n}{\sqrt{n}} \qquad \left(z \in \Omega, \, g_\Omega(z,\infty)>\frac{1}{n}\right)
$$
where $w_1^n,w_2^n,\dots,w_n^n$ are the points of a Fekete $n$-tuple for $E$ and $\nu$ is the equilibrium measure for $E$.

\begin{theorem}
\label{thmExplicit2}
Let $s>0$, and suppose that $n$ is sufficiently large so that the following conditions hold :

\begin{enumerate}[\rm(i)]
\item $\displaystyle \frac{1}{n} \leq  s$
\item $\displaystyle C(E) \frac{\log n}{\sqrt{n}} \leq \frac{s}{4}$
\item $\displaystyle e^{ns/4} \geq \frac{R(E)e^s}{r(E)}$
\item $\displaystyle \left(\frac{R(E)e^s}{r(E)}\right)^{1/n} \frac{\delta_ n(E)}{\operatorname{cap}(E)} \leq e^{s/2}$.
\end{enumerate}
Then the polynomial
$$P_n(z)=P_{n,s}(z):=z\frac{e^{-ns/2}}{\operatorname{cap}(E)^n} \prod_{j=1}^n (z-w_j^n)$$
satisfies
$$E \subset \operatorname{int} (\mathcal{K}(P_n)) \subset E_s,$$
where $E_s:= E \cup \{z \in \Omega : g_\Omega(z,\infty) \leq s\}$.

\end{theorem}

In particular, the set $E$ is totally approximable by the filled Julia sets of the polynomials $P_n$'s, since the compact sets $E_s$ shrink to $E$ as $s$ decreases to $0$.

\begin{proof}

The idea is very similar to the proof of Theorem \ref{thmExplicit}.

Let $\Omega_s := \RiemannSphere \setminus E_s = \{z \in \Omega : g_{\Omega}(z,\infty)>s\}$. Note that $$\RiemannSphere \setminus \overline{\mathbb{D}}(0,R(E)e^s) \subset \Omega_s,$$ by Corollary \ref{CorBern}.

Now, let $z \in \Omega_s$. Then $g(z,\infty) > 1/n$ by (i), so that by Theorem \ref{thmPRI},
$$\left| \frac{1}{n} \sum_{j=1}^n \log{|z-w_j^n|} - p_\nu(z) \right| \leq C(E) \frac{\log n}{\sqrt{n}} \leq \frac{s}{4},$$
where we used condition (ii). By Theorem \ref{TheoFrostman}, we get
$$g_\Omega(z,\infty) -\frac{s}{4} \leq \frac{1}{n} \sum_{j=1}^n \log{|z-w_j^n|} - I(\nu)$$
and thus
$$\exp \left(n \left(g_\Omega(z,\infty) - \frac{s}{2} - \frac{s}{4} \right) \right) \leq |Q_n(z)|,$$
where
$$Q_n(z):= \frac{e^{-ns/2}}{\operatorname{cap}(E)^n} \prod_{j=1}^n (z-w_j^n) = \frac{P_n(z)}{z}.$$
Since $g_\Omega(z,\infty) > s$, we obtain
\begin{equation}
\label{eq}
\frac{R(E)e^s}{r(E)} \leq e^{ns/4} < |Q_n(z)|,
\end{equation}
by (iii), and this holds for all $z \in \Omega_s$.

On the other hand, if $z \in E$, then by Theorem \ref{thmFek} and (iv),
\begin{equation}
\label{eqq}
|Q_n(z)| \leq e^{-ns/2}\left(\frac{\delta_n(E)}{\operatorname{cap}(E)}\right)^n \leq e^{-ns/2} e^{ns/2} \frac{r(E)}{R(E)e^s} = \frac{r(E)}{R(E)e^s}.
\end{equation}
Now, for $z \in \Omega_s$, we have, by (\ref{eq}),
$$|P_n(z)| = |z| |Q_n(z)| > r(E) \frac{R(E)e^s}{r(E)} = R(E)e^s$$
so that $P_n(z) \in \RiemannSphere \setminus \overline{\mathbb{D}}(0,R(E)e^s) \subset \Omega_s$. Moreover, again by (\ref{eq}),
$$|P_n(z)| = |z| |Q_n(z)| > |z| \frac{R(E)e^s}{r(E)},$$
where
$$\frac{R(E)e^s}{r(E)}>1.$$ It follows that the iterates of $P_n$ converge to $\infty$ on $\Omega_s$, so that $\Omega_s \subset \RiemannSphere \setminus \mathcal{K}(P_n)$ and thus $\mathcal{K}(P_n) \subset E_s$.

Finally, for $z \in E$, we have, by (\ref{eqq}),

$$|P_n(z)| = |z| |Q_n(z)| < R(E)e^s \frac{r(E)}{R(E)e^s} = r(E)$$
and thus $P_n(E) \subset \mathbb{D}(0,r(E)) \subset E$. Clearly, this implies that $E \subset \operatorname{int} (\mathcal{K}(P_n))$, which completes the proof of the theorem.

\end{proof}

We can now use Theorem \ref{thmExplicit2} to obtain a precise estimate for the rate of approximation. Recall from the introduction that $s_n(E)$, $n \in \mathbb{N}$, is defined to be the infimum of $s>0$ for which there exists a polynomial $p_n$ of degree $n$ such that
$$E \subset \mathcal{K}(p_n) \subset E_s,$$
where $E_s:= E \cup \{z \in \Omega : g_\Omega(z,\infty) \leq s\}$.

Solving for $s$ in Theorem \ref{thmExplicit2}, one easily obtains the existence of a constant $c'=c'(E)$ such that
$$s_n(E) \leq c' \left( \frac{\log{n}}{\sqrt{n}} + \log{\frac{\delta_n(E)}{\operatorname{cap}(E)}} \right) \qquad (n \geq 1).$$

It follows from this and Theorem \ref{thmPRI2} that there is a constant $c=c(E)$ depending only on $E$ such that
$$s_n(E) \leq c \, \frac{\log{n}}{\sqrt{n}} \qquad (n \geq 1),$$
thereby proving Theorem \ref{mainthm2}.

\section{Hilbert's Lemniscate Theorem}
\label{sec5}

In this section, we give a proof of Theorem \ref{mainthm1} using a classical result of Hilbert on the approximation of planar sets by polynomial lemniscates. We first present a proof of this latter result based on Fekete polynomials.

\begin{theorem}[Hilbert's Lemniscate Theorem \cite{HIL}]
\label{Hilbert}
Let $E \subset \mathbb{C}$ be a compact set with connected complement $\Omega$, and let $U$ be an open neighborhood of $E$. Then there exists a polynomial $Q$ such that
$$|Q(z)| \leq 1 \qquad (z \in E)$$
and
$$|Q(z)|>1 \qquad (z \in \mathbb{C} \setminus U).$$
\end{theorem}

\begin{proof}
Again, we can assume that $\operatorname{cap}(E)>0$. Define
$$A:= \inf_{z \in \RiemannSphere \setminus U} g_\Omega(z,\infty)$$
and
$$B:= \sup_{z \in \RiemannSphere \setminus U} \tau_\Omega(z,\infty),$$
so that $A>0$ and $B<\infty$. By Theorem \ref{thmBernstein}, if $q$ is a Fekete polynomial for $E$ of degree $n$, then

$$\left( \frac{|q(z)|}{\|q\|_E}\right)^{1/n} \geq e^{A} \left( \frac{\operatorname{cap}(E)}{\delta_n(E)} \right)^{B} \qquad (z \in \mathbb{C} \setminus U).$$
Here we used the fact that $\delta_n(E) \geq \operatorname{cap}(E)$ for all $n$, cf. Theorem \ref{thmFekSze}. Finally, since $\delta_n(E) \to \operatorname{cap}(E)$ as $n \to \infty$, the right-hand side will exceed $1$ for all sufficiently large $n$, and setting $Q:=q/\|q\|_E$ then gives the desired polynomial $Q$.

\end{proof}

We now use Theorem \ref{Hilbert} to give another proof of Theorem \ref{mainthm1}.

\begin{proof}

Again, we can assume without loss of generality that $0$ belongs to $E$ and that it is an interior point.

Let $U$ be a bounded neighborhood of $E$, and let $Q$ be the polynomial as in Hilbert's Theorem \ref{Hilbert}. Dividing $Q$ by a constant slightly larger than one if necessary, we can assume that $\|Q\|_E < 1$.

Let $R,r>0$ such that $\mathbb{D}(0,r) \subset E$ and $U \subset \mathbb{D}(0,R)$. For $k \geq 1$, define the polynomial $P_k(z):=zQ^k(z)$. Here $Q^k$ denotes the $k$-th power of the polynomial $Q$. Then for $k$ sufficiently large, we have
$$P_k(E) \subset \mathbb{D}(0,r) \subset E$$
and
$$P_k\left(\RiemannSphere \setminus U \right) \subset \RiemannSphere \setminus \mathbb{D}(0,R) \subset \RiemannSphere \setminus U,$$
since $P_k \to 0$ uniformly on $E$ and $P_k \to \infty$ uniformly on $\RiemannSphere \setminus U$.

This implies that $E \subset \mathcal{K}(P_k)$. Also, the inequality
$$|P_k(z)| \geq |z| \min_{\RiemannSphere \setminus U}|Q| \qquad (z \in \RiemannSphere \setminus U),$$
where $\displaystyle{\min_{\RiemannSphere \setminus U}|Q|>1}$, shows that the iterates of $P_k$ tend to $\infty$ on $\RiemannSphere \setminus U$, so that $\RiemannSphere \setminus U \subset \RiemannSphere \setminus \mathcal{K}(P_k)$, i.e. $\mathcal{K}(P_k) \subset U$.

Since $U$ can be made arbitrarily close to $E$, the result follows.

\end{proof}

\section{Totally approximable sets}
\label{sec6}

In this section, we prove Theorem \ref{mainthm3}, which states that a nonempty proper subset $S$ of the complex plane $\mathbb{C}$ is totally approximable by polynomial filled Julia sets if and only if $S$ is bounded and $\mathbb{C} \setminus \operatorname{int}(S)$ is connected.

First, we need an approximation result.

\begin{proposition}
\label{TotApproxProp}
Let $S$ be a nonempty proper subset of the complex plane $\mathbb{C}$. If $S$ is bounded and $\mathbb{C} \setminus \operatorname{int}(S)$ is connected, then $S$ is totally approximable by a collection $\{K^\epsilon\}$ of nonempty compact sets with connected complement.

\end{proposition}

The proof of Proposition \ref{TotApproxProp} uses the following classical result from plane topology, known as Janiszewski's Theorem.

\begin{lemma}
\label{LemmaJan}
Let $A$ and $B$ be closed sets in $\RiemannSphere$ such that $A \cap B$ is connected. If $A$ and $B$ both have connected complement, then $A \cup B$ also has connected complement.
\end{lemma}

\begin{proof}
See \cite[Theorem 1.9]{POM}.
\end{proof}

We can now proceed with the proof of Proposition \ref{TotApproxProp}. For a set $A$ and $\epsilon>0$, we denote by $U_\epsilon(A)$ the open $\epsilon$-neighborhood of the set $A$.

\begin{proof}

For $\epsilon>0$, let $G^\epsilon$ be a grid of squares in $\mathbb{C}$ of sidelength $\epsilon/2$. Let $H^\epsilon$ be the union of the closed squares in $G^\epsilon$ which intersect $\partial S$. Note that since $S$ is bounded, the set $H^\epsilon$ consists of finitely many closed squares and is therefore compact. Now, let $F^\epsilon$ be defined by taking each square $Q$ in $H^\epsilon$ and replacing it by a smaller square $\tilde{Q}$ with the same center of sidelength $\epsilon/4$. Then it is easy to see that the complement of $F^\epsilon$ is connected.

Now, define a collection of compact sets $\{K^\epsilon\}$ by
$$K^\epsilon:= E^\epsilon \cup F^\epsilon,$$
where $E^\epsilon:= \mathbb{C} \setminus U_\epsilon(\mathbb{C} \setminus \operatorname{int}(S))$. Note that the sets $E^\epsilon$ and $F^\epsilon$ are disjoint, since $F^\epsilon$ is contained in the $(\epsilon/\sqrt{2})$-neighborhood of $\partial S$. See Figure 2.

\begin{figure}[h!t!b]
\label{fig1}
\begin{center}

\includegraphics[width=.6\linewidth]{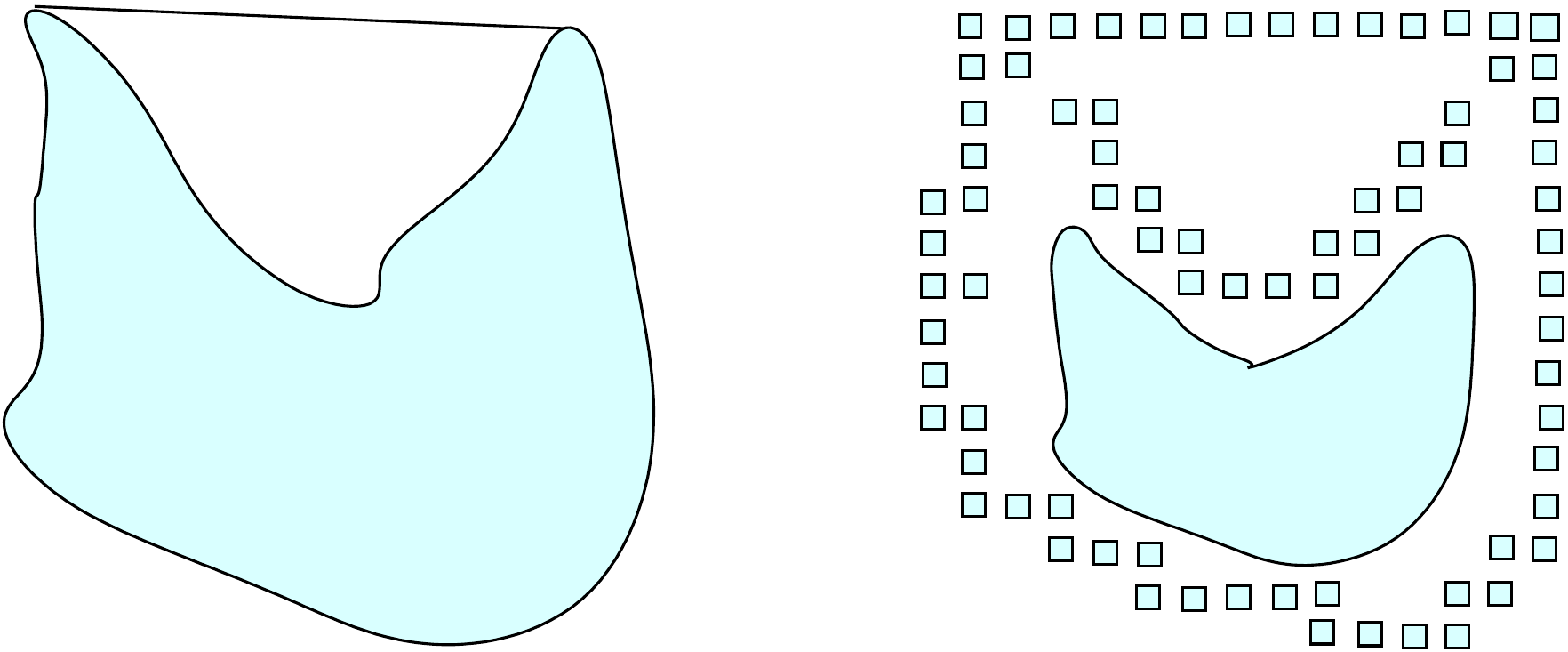}
\label{f:fig1}
\caption{On the left, a nonempty bounded subset $S$ of the plane whose interior has connected complement. On the right, the corresponding approximating set $K^\epsilon$. }
\end{center}

\end{figure}

We claim that for every $\epsilon>0$, the complement of $K^\epsilon$ is connected. Indeed, first note that the complement of $E^\epsilon$ is connected. To see this, note that
$$\mathbb{C} \setminus E^\epsilon = \bigcup_{z \notin \operatorname{int}(S)} \mathbb{D}(z,\epsilon),$$
hence each component of $\mathbb{C} \setminus E^\epsilon$ contains a disk $\mathbb{D}(z,\epsilon)$ for some $z \notin \operatorname{int}(S)$ and therefore must intersect $\mathbb{C} \setminus \operatorname{int}(S)$. Since $\mathbb{C} \setminus \operatorname{int}(S)$ is connected and $\mathbb{C} \setminus E^\epsilon \supset \mathbb{C} \setminus \operatorname{int}(S)$, it follows that each component of $\mathbb{C} \setminus E^\epsilon$ has to contain $\mathbb{C} \setminus \operatorname{int}(S)$. Thus there can only be one component; in other words, the set $\mathbb{C} \setminus E^\epsilon$ is connected. Finally, since $F^\epsilon$ also has connected complement and $E^\epsilon, F^\epsilon$ are disjoint, we can apply Lemma \ref{LemmaJan} to deduce that $K^\epsilon$ also has connected complement.

Next, we claim that $S$ is totally approximable by the sets $K^\epsilon$. It suffices to prove that  $d(\partial S, \partial K^\epsilon) \to 0$ and $d(S,K^\epsilon) \to 0$ as $\epsilon \to 0$.

First, assume for a contradiction that $d(\partial S, \partial K^{\epsilon_n})> \delta$ for all $n$, for some $\delta>0$ and some sequence $\epsilon_n$ strictly decreasing to $0$. Then for each $n$, either $\partial S \nsubseteq U_\delta(\partial K^{\epsilon_n})$ or $\partial K^{\epsilon_n} \nsubseteq U_\delta(\partial S)$. However, note that if $z \in \partial S$, then $z$ belongs to a square $Q$ in $H^\epsilon$, so that the distance between $z$ and the boundary of the corresponding square $\tilde{Q}$ in $F^\epsilon$ is less than $\epsilon/(4\sqrt{2})$, hence $\partial S \subset U_\epsilon(\partial K^\epsilon)$, for all $\epsilon>0$. We can therefore assume that the second case, $\partial K^{\epsilon_n} \nsubseteq U_\delta(\partial S)$, holds for all $n$. Now, recall that $F^\epsilon$ is contained in the $(\epsilon/\sqrt{2})$-neighborhood of $\partial S$. It follows that for all sufficiently large $n$, we have $\partial E^{\epsilon_n} \nsubseteq U_\delta(\partial S)$. Then for each such $n$, there is a point $z_n \in \partial E^{\epsilon_n} \subset \operatorname{int}(S)$ such that the disk $\mathbb{D}(z_n,\delta)$ does not intersect the boundary of $S$, hence must be contained in $\operatorname{int}(S)$. Passing to a subsequence if necessary, we can assume that $z_n \to z_0 \in \overline{S}$. It is easy to see that $z_0$ necessarily belongs to $\partial S$, since the sets $E^{\epsilon_n}$ form a compact exhaustion of $\operatorname{int}(S)$ with $E^{\epsilon_n} \subset \operatorname{int}(E^{\epsilon_{n+1}})$ for all $n$. This contradicts the fact that for sufficiently large $n$, the disk $\mathbb{D}(z_0,\delta/2)$ is contained in the disk $\mathbb{D}(z_n,\delta)$, a subset of $\operatorname{int}(S)$. Therefore $d(\partial S, \partial K^\epsilon) \to 0$ as $\epsilon \to 0$.

It remains to prove that $d(S,K^\epsilon) \to 0$ as $\epsilon \to 0$. Again, assume for a contradiction that $d(S,K^{\epsilon_n})>\delta$ for all $n$, for some $\delta>0$ and some sequence $\epsilon_n$ strictly decreasing to $0$. Since $E^\epsilon \subset \operatorname{int}(S)$ and $F^\epsilon$ is contained in the $(\epsilon/\sqrt{2})$-neighborhood of $\partial S$, we have $K^\epsilon \subset U_\epsilon(S)$, for all $\epsilon>0$. We can therefore assume that for each $n$, $S \nsubseteq U_\delta(K^{\epsilon_n})$. Then for each $n$, there exists $z_n \in S$ such that the disk $\mathbb{D}(z_n,\delta)$ does not intersect $K^{\epsilon_n}$. Assume without loss of generality that $z_n \to z_0 \in \overline{S}$. Then for all sufficiently large $n$, the disk $\mathbb{D}(z_0,\delta/2)$ does not intersect $K^{\epsilon_n}$. If $z_0 \in \partial S$, then for sufficiently large $n$, the disk $\mathbb{D}(z_0,\delta/2)$ contains a square in $F^{\epsilon_n}$, and we get a contradiction. If $z_0 \in \operatorname{int}(S)$, we get a contradiction with the fact that the sets $E^{\epsilon_n}$ form a compact exhaustion of $\operatorname{int}(S)$. This completes the proof of the proposition.

\end{proof}

Now, recall from Theorem \ref{mainthm1} that any nonempty compact set with connected complement is totally approximable by polynomial filled Julia sets. Combining this with Proposition \ref{TotApproxProp}, we get that any nonempty bounded set whose interior has connected complement is also totally approximable by polynomial filled Julia sets. This proves the converse implication in Theorem \ref{mainthm3}.

For the proof of the other implication, we need the following elementary lemmas.

\begin{lemma}
\label{lem1}
For any set $A \subset \mathbb{C}$, we have
$$U_\epsilon(A) = A \cup U_\epsilon(\partial A).$$
\end{lemma}

\begin{proof}
Clearly, we have $A \subset U_\epsilon(A)$. Also, if $z \in U_\epsilon(\partial A)$, then $\mathbb{D}(z,\epsilon) \cap \partial A \neq \emptyset$, so that $\mathbb{D}(z,\epsilon) \cap A \neq \emptyset$, from which we deduce that $z \in U_\epsilon(A)$. This shows that $A \cup U_\epsilon(\partial A) \subset U_\epsilon(A)$.

For the other inclusion, if $z \in U_\epsilon(A)$ but $z \notin A$, then the disk $\mathbb{D}(z,\epsilon)$ intersects both $A$ and $\mathbb{C} \setminus A$. By connectedness, it must intersect $\partial A$, so that $z \in U_\epsilon(\partial A)$.

\end{proof}

\begin{lemma}
\label{lem2}
Let $A,B$ be two nonempty proper subsets of the plane with $d(A,B)<\epsilon$ and $d(\partial A, \partial B) < \epsilon$. Then $d(\mathbb{C} \setminus A, \mathbb{C} \setminus B) < 2\epsilon$.
\end{lemma}

\begin{proof}
We first show that $\mathbb{C} \setminus A \subset U_{2\epsilon}(\mathbb{C} \setminus B)$.

Let $z \in \mathbb{C} \setminus A$. If $z \notin U_\epsilon(A)$, then $z \in \mathbb{C} \setminus B$, since $B \subset U_\epsilon(A)$. On the other hand, if $z \in U_\epsilon(A)$, then by Lemma \ref{lem1}, we have that $z \in U_\epsilon(\partial A) \subset U_{2\epsilon}(\partial B)$. It follows that $\mathbb{D}(z,2\epsilon) \cap \partial B \neq \emptyset$, so that $\mathbb{D}(z,2\epsilon) \cap (\mathbb{C} \setminus B) \neq \emptyset$, from which we deduce that $z \in U_{2\epsilon}(\mathbb{C} \setminus B)$. In both cases, we get that $z \in U_{2 \epsilon}(\mathbb{C} \setminus B)$.

This proves that $\mathbb{C} \setminus A \subset U_{2\epsilon}(\mathbb{C} \setminus B)$. The same argument with $A$ and $B$ interchanged yields the inclusion $\mathbb{C} \setminus B \subset U_{2\epsilon}(\mathbb{C} \setminus A)$, hence $d(\mathbb{C} \setminus A, \mathbb{C} \setminus B) < 2\epsilon$, as required.

\end{proof}

We can now proceed with the proof of Theorem \ref{mainthm3}.

\begin{proof}[Proof of Theorem \ref{mainthm3}]

Let $S$ be a nonempty proper subset of $\mathbb{C}$.

As previously mentioned, if $S$ is bounded and $\mathbb{C} \setminus \operatorname{int}(S)$ is connected, then by Proposition \ref{TotApproxProp} and Theorem \ref{mainthm1}, the set $S$ is totally approximable by polynomial filled Julia sets.

Conversely, assume that $S$ is totally approximable by polynomial filled Julia sets. Then for any $\epsilon>0$, there exists a polynomial $P_\epsilon$ such that $d(S,\mathcal{K}(P_\epsilon))<\epsilon$ and $d(\partial S, \mathcal{J}(P_\epsilon)) < \epsilon$. Clearly, this implies that $S$ is bounded, since for each $\epsilon>0$, the set $\mathcal{K}(P_\epsilon)$ is compact.

It remains to prove that $\mathbb{C} \setminus \operatorname{int}(S)$ is connected. Assume for a contradiction that $\mathbb{C} \setminus \operatorname{int}(S) = E \cup F$, where $E$ and $F$ are disjoint nonempty closed subsets of $\mathbb{C}$. Since $S$ is bounded, one of the sets $E,F$ must be unbounded, say $E$, while the other set, $F$, is bounded. Now, fix a point $z_0 \in F$, and let $\delta>0$ sufficiently small so that $\overline{\mathbb{D}}(z_0,2\delta) \subset \mathbb{C} \setminus E$. Since $F \cup \overline{\mathbb{D}}(z_0,2\delta)$ is a compact subset of the open set $\mathbb{C} \setminus E$, there is a bounded open set $V$ with $F \cup \overline{\mathbb{D}}(z_0,2\delta) \subset V \subset \overline{V} \subset \mathbb{C} \setminus E$. In particular, the boundary of $V$ is contained in $\mathbb{C} \setminus (E \cup F) = \operatorname{int}(S)$.

Now, let $\epsilon>0$ be smaller than both $\delta$ and half the distance between $\partial V$ and $E \cup F$. Since $z_0 \in F \subset \mathbb{C} \setminus \operatorname{int}(S)$, we have that either $z_0 \in \partial S$ or $z_0 \in \mathbb{C} \setminus S$. In both cases, since $\partial S \subset U_\epsilon(\mathcal{J}(P_\epsilon))$ and $\mathbb{C} \setminus S \subset U_{2\epsilon}(\mathbb{C} \setminus \mathcal{K}(P_\epsilon))$ by Lemma \ref{lem2}, we get that the disk $\mathbb{D}(z_0,2\epsilon)$ contains a point $w_0 \in \mathbb{C} \setminus \mathcal{K}(P_\epsilon)$.

This is enough to obtain a contradiction. Indeed, since $w_0 \in \mathbb{D}(z_0,2\delta) \subset V$ and $w_0 \in \mathbb{C} \setminus \mathcal{K}(P_\epsilon)$, the connectedness of $\mathbb{C} \setminus \mathcal{K}(P_\epsilon)$ (see \cite[Lemma 9.4]{MIL}) implies that $\partial V$ must have non-empty intersection with $\mathbb{C} \setminus \mathcal{K}(P_\epsilon)$. Since $\mathbb{C} \setminus \mathcal{K}(P_\epsilon) \subset U_{2\epsilon}(\mathbb{C} \setminus S)$, again by Lemma \ref{lem2}, it follows that the distance between $\partial V$ and $\mathbb{C} \setminus \operatorname{int}(S) = E \cup F$ is less than $2 \epsilon$, a contradiction.

Therefore, $\mathbb{C} \setminus \operatorname{int}(S)$ is connected, as required.

\end{proof}

%%%%%
\section{Approximation of connected sets}
\label{sec7}

In this section, we give a proof of Theorem \ref{t:connected} by showing that our approximation method gives quasicircle Julia sets if the original set $E$ is connected.

\begin{theorem}
\label{theoconnected}
Let $E$ be a connected compact set with connected complement, and assume as before that $0 \in \operatorname{int}(E)$. Then for any $s>0$ and $n \in \mathbb{N}$ such that $E \subset \operatorname{int}(\mathcal{K}(P_{n,s}))$, where $P_{n,s}$ is as defined in Theorem \ref{thmExplicit}, the Julia set $\mathcal{J}(P_{n,s})$ is a Jordan curve. Moreover, the polynomial $P_{n,s}$ is hyperbolic and $\mathcal{J}(P_{n,s})$ is a quasicircle.
\end{theorem}

\begin{proof}[Proof of Theorem \ref{t:connected}]
First, to prove that $\mathcal{J}(P_{n,s})$ is a Jordan curve, by \cite[Theorem VI.5.3]{CAR}  it suffices to show  that the Fatou set $\RiemannSphere \setminus \mathcal{J}(P_{n,s})$ consists of two completely invariant components. Since $E$ is connected and is contained in the Fatou set, it must be contained in a single Fatou component, say $A$. For polynomials, the basin of attraction of $\infty$ is completely invariant, so it suffices to show that $A$ is also completely invariant.

Since $A$ contains the fixed point $0$, we have $P_{n,s}(A) \subset A$.  If $P_{n,s}^{-1}(A)$ is not contained in $A$, then there exists a Fatou component distinct from $A$, say $A^{\prime}$, which is mapped onto $A$ by $P_{n,s}$.
This is impossible, since all the zeros of $P_{n,s}$ belong to $E \subset A$. Therefore, $A$ is completely invariant and the Julia set is a Jordan curve.

Now, since $A$ is completely invariant, all the critical points of $P_{n,s}$ belong to $A$, by \cite[Theorem V.1.3]{CAR}. It then follows from \cite[Theorem V.2.2]{CAR})
that $P_{n,s}$ is hyperbolic.  Finally, by \cite[Theorem VI.2.1]{CAR}, the Julia set $\mathcal{J}(P_{n,s})$ is not only a Jordan curve, but also a quasicircle.

\end{proof}

\section{Numerical examples}
\label{sec8}
In this section, we illustrate the method of Section \ref{sec3} with some numerical examples.

As shown in Theorem \ref{thmExplicit2}, the computation of Julia sets approximating a given compact set $E$ requires a numerical method for the computation of Fekete points for $E$. Unfortunately, the numerical computation of high-degree Fekete points is a hard large-scale optimization problem. Indeed, as far as we know, even for simple compact sets such as triangles, Fekete points have been computed only up to relatively small degrees, see e.g. \cite{TAY}, \cite{PAS} and the references therein. In order to circumvent this issue, alternative numerical methods have been introduced for the computation of approximate Fekete points : see e.g. \cite{SOM} for a method based on QR factorization of Vandermonde matrices with column pivoting, and \cite{BOS2} for a discussion of the theoretical aspects of this method.

However, for slightly complicated sets, we were not able to use the algorithm of \cite{SOM} to compute Fekete points of high degree, due to rank-deficiency and ill-conditioning issues. For this reason, we decided to replace Fekete points by other equidistributed points, Leja points, which are much more convenient from a computational point of view. This sequence of points, introduced by Leja \cite{LEJ}, is defined inductively as follows.

Choose an arbitrary point $z_1 \in \partial E$ and for $n \geq 2$, choose $z_n \in \partial E$ by maximizing the product
$$\prod_{j=1}^{n-1}|z-z_j|$$
for $z \in E$, i.e.
$$\prod_{j=1}^{n-1}|z_n-z_j| = \max_{z \in E} \prod_{j=1}^{n-1}|z-z_j|.$$

It is well-known that Leja points are easier to compute numerically than Fekete points, mainly because they are defined inductively by a univariate optimization problem. Furthermore, this optimization problem can be made discrete by replacing the compact set $E$ by a sufficiently refined discretization.

We shall prove below that our approximation theorem \ref{thmExplicit2} remains valid if Fekete points are replaced by Leja points, with small modifications. First, we need some lemmas.

As previously mentioned, Leja points are also equidistributed with respect to harmonic measure on $\partial E$, in the sense of Proposition \ref{weakconvergence}. In fact, Pritsker's Theorem \ref{thmPRI} also holds for Leja points, with the same constant.

\begin{lemma}
If $(z_n)$ is a sequence of Leja points for $E$, then
\begin{equation*}
\label{eqprit}
\left| \frac{1}{n} \sum_{j=1}^n \log{|z-z_j|} - p_\nu(z) \right| \leq C(E) \frac{\log n}{\sqrt{n}} \qquad \left(z \in \Omega, \, g_\Omega(z,\infty)>\frac{1}{n}\right)
\end{equation*}
where $C(E)$ is Pristker's constant as in Definition \ref{defPRI}.
\end{lemma}

\begin{proof}
See \cite[Corollary 2.3]{PRI}.

\end{proof}

We will also need an analogue of Theorem \ref{thmFekSze} for Leja points.

\begin{lemma}[Leja]
\label{thmLeja}
For $n \in \mathbb{N}$, define
$$a_n:= \prod_{j=1}^{n}|z_{n+1}-z_j|$$
Then
$a_n^{1/n} \geq \operatorname{cap}(E)$ for each $n \in \mathbb{N}$, and $a_n^{1/n} \to \operatorname{cap}(E)$ as $n \to \infty$.
\end{lemma}

\begin{proof}
See \cite[Lemme 1]{LEJ}.

\end{proof}

We can now prove the following variant of Theorem \ref{thmExplicit2} for Leja points. As before, we assume that $E \subset \mathbb{C}$ is a uniformly perfect compact set with connected complement $\Omega$, and that $0 \in E$ is an interior point. Recall then that $r(E)$ and $R(E)$ are defined by
$$r(E):= \operatorname{dist}(0,\partial E)$$
and
$$R(E):= \sup_{w \in E} |w|.$$

\begin{theorem}[Leja points version]
\label{thmExplicitLeja}
Let $s>0$, and suppose that $n$ is sufficiently large so that the following conditions hold :

\begin{enumerate}[\rm(i)]
\item $\displaystyle \frac{1}{n} \leq  s$
\item $\displaystyle C(E) \frac{\log n}{\sqrt{n}} \leq \frac{s}{4}$
\item $\displaystyle e^{ns/4} \geq \frac{R(E)e^s}{r(E)}$
\item $\displaystyle \left(\frac{R(E)e^s}{r(E)}\right)^{1/n} \frac{a_n^{1/n}}{\operatorname{cap}(E)} \leq e^{s/2}$.
\end{enumerate}
Then the polynomial
$$\tilde{P}_n(z)=\tilde{P}_{n,s}(z):=z\frac{e^{-ns/2}}{\operatorname{cap}(E)^n} \prod_{j=1}^n (z-z_j)$$
satisfies
$$E \subset \operatorname{int}(\mathcal{K}(\tilde{P}_n)) \subset E_s,$$
where $E_s:= E \cup \{z \in \Omega : g_\Omega(z,\infty) \leq s\}$.

\end{theorem}

\begin{remark}
The only difference between the assumptions of Theorem \ref{thmExplicitLeja} and those of Theorem \ref{thmExplicit2} is condition (iv), where $\delta_n(E)$ is replaced by $a_n^{1/n}$.
\end{remark}

\begin{proof}
The proof is exactly the same as in Theorem \ref{thmExplicit2} except that in inequality (\ref{eqq}), the use of Theorem \ref{thmFek} is replaced by the inequality
$$\prod_{j=1}^n |z-z_j|\leq \prod_{j=1}^n |z_{n+1}-z_j|=a_n \qquad (z \in E),$$
which follows directly from the definition of Leja points.

\end{proof}

We mention that using Leja points instead of Fekete points does not yield a better rate of approximation in Theorem \ref{mainthm2}. The main improvement, however, is that the polynomials $\tilde{P}_n$ can be computed numerically in a reasonable amount of time, even for high values of $n$, thereby giving an efficient algorithm to approximate a given shape by a polynomial Julia set.

More precisely, let $E$ be a uniformly perfect compact set with connected complement, and assume that $0 \in \operatorname{int}(E)$. Take $s>0$ small and $n \in \mathbb{N}$ large. The algorithm consists of the following steps.

\begin{itemize}
\item Compute $n$ Leja points for $E$ using a sufficiently refined discretization of the set.
\item Compute the polynomial
$$\tilde{P}_n(z)=z\frac{e^{-ns/2}}{\operatorname{cap}(E)^n} \prod_{j=1}^n (z-z_j),$$
where $\operatorname{cap}(E)$ is approximated by the quantity $a_n^{1/n}$, in view of Lemma \ref{thmLeja}.
\item Plot the filled Julia set of the polynomial $\tilde{P}_n$.
\end{itemize}

The first two steps are easily carried out with \textsc{matlab}, for instance, even for large values of $n$. As for plotting the filled Julia set, it can be done in the obvious natural way, that is, dividing a square containing the set $E$ into a sufficiently large number of small squares (the pixels), iterating the center of each pixel a fixed number of times, and then coloring the pixel depending on whether the absolute value of the last iterate is big or not. This works relatively well in general, at least when $n$ is not too large, say $n \leq 2000$. We also tried a more precise method for plotting Julia sets based on distance estimation (see e.g. \cite{DRA}), but the difference in quality of the images obtained was negligible compared to the longer time required for the computations.

We also mention that another difficulty in the implementation of the algorithm is to find how small $s$ can be taken, given $n$. By Theorem \ref{mainthm2}, we know that the ratio between the best possible $s$ and $\log{n}/\sqrt{n}$ is bounded. However, we observed that using $s=\log{n}/\sqrt{n}$ generally gives poor results compared to smaller values, say $s=1/n$.

We now present several numerical examples to illustrate the method. We thank Siyuan Li, Xiao Li and Ryan Pachauri, three undergraduate students at the University of Washington who produced some of the examples as part of a research project supervised by the second author for the Washington Experimental Mathematics Lab (WXML).

\newpage

\begin{example}
{\em The rabbit Julia set.}

Our first example illustrates Theorem \ref{theoconnected}, which says that the Julia sets of the approximating polynomials are Jordan curves if the original set $E$ is connected.

Figure 3 is a representation of the filled Julia set of the polynomial
$$\tilde{P}_n(z)=z\frac{e^{-ns/2}}{\operatorname{cap}(E)^n} \prod_{j=1}^n (z-z_j),$$
with $n=700$ and $s=1/700$, as well as the original set $E$ (boundary in black). The Julia set was plotted using a resolution of $5000 \times 5000$, and the computations took approximately $120$ seconds.

\begin{figure}[h!t!b]
\begin{center}

\includegraphics[width=.7\textwidth]{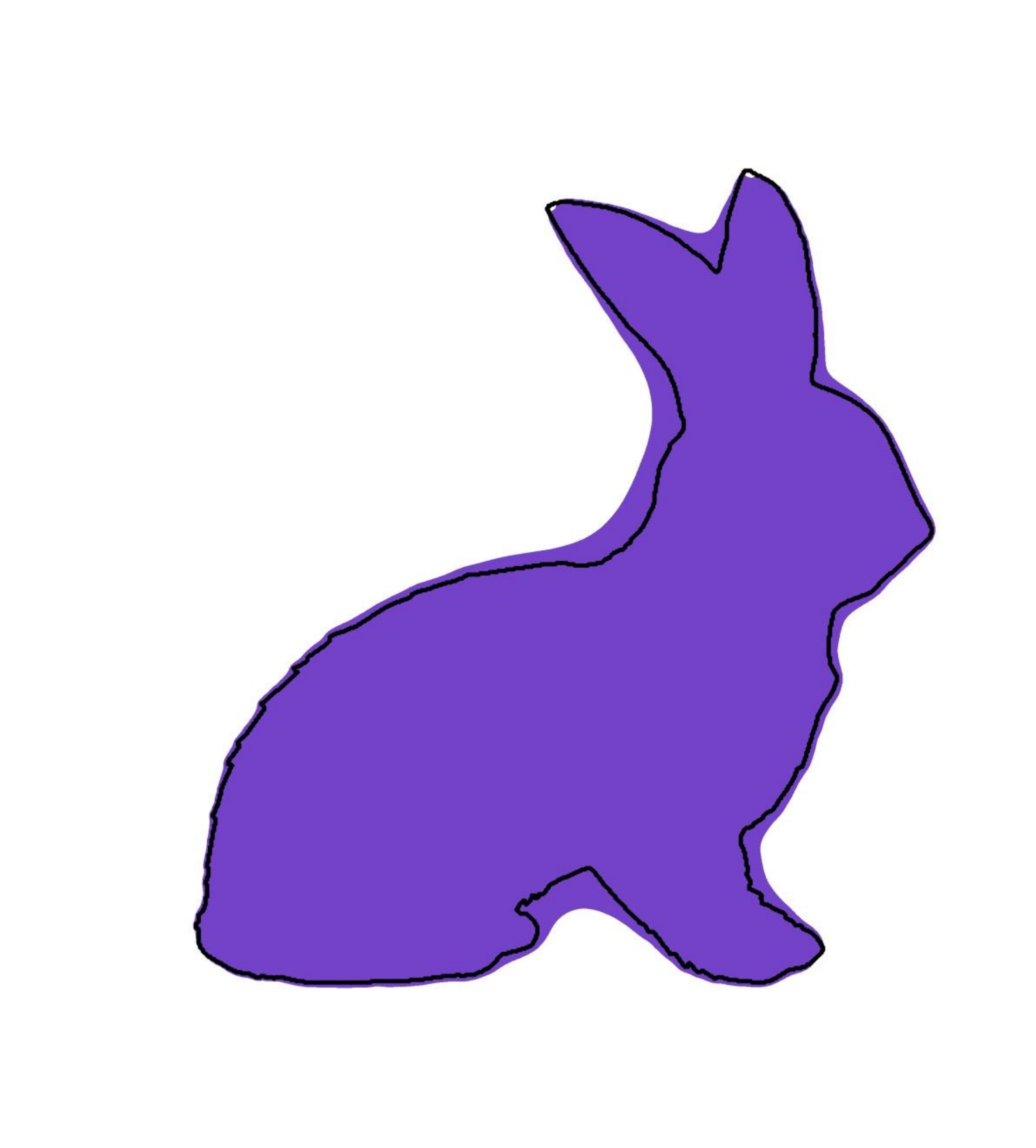}
\caption{A filled Julia set in the shape of a rabbit.}
\end{center}

\end{figure}

\end{example}

\newpage

\begin{example}
{\em The Batman Julia set.}

Here is another example of a Jordan curve Julia set approximating a connected shape.

Figure 4 is a representation of the filled Julia set of the polynomial $\tilde{P}_n$
with $n=700$ and $s=1/700$, as well as the original set $E$ (boundary in black). The Julia set was plotted using a resolution of $6000 \times 6000$, and the computations took approximately $187$ seconds.

\begin{figure}[h!t!b]
\begin{center}

\includegraphics[width=.8\textwidth]{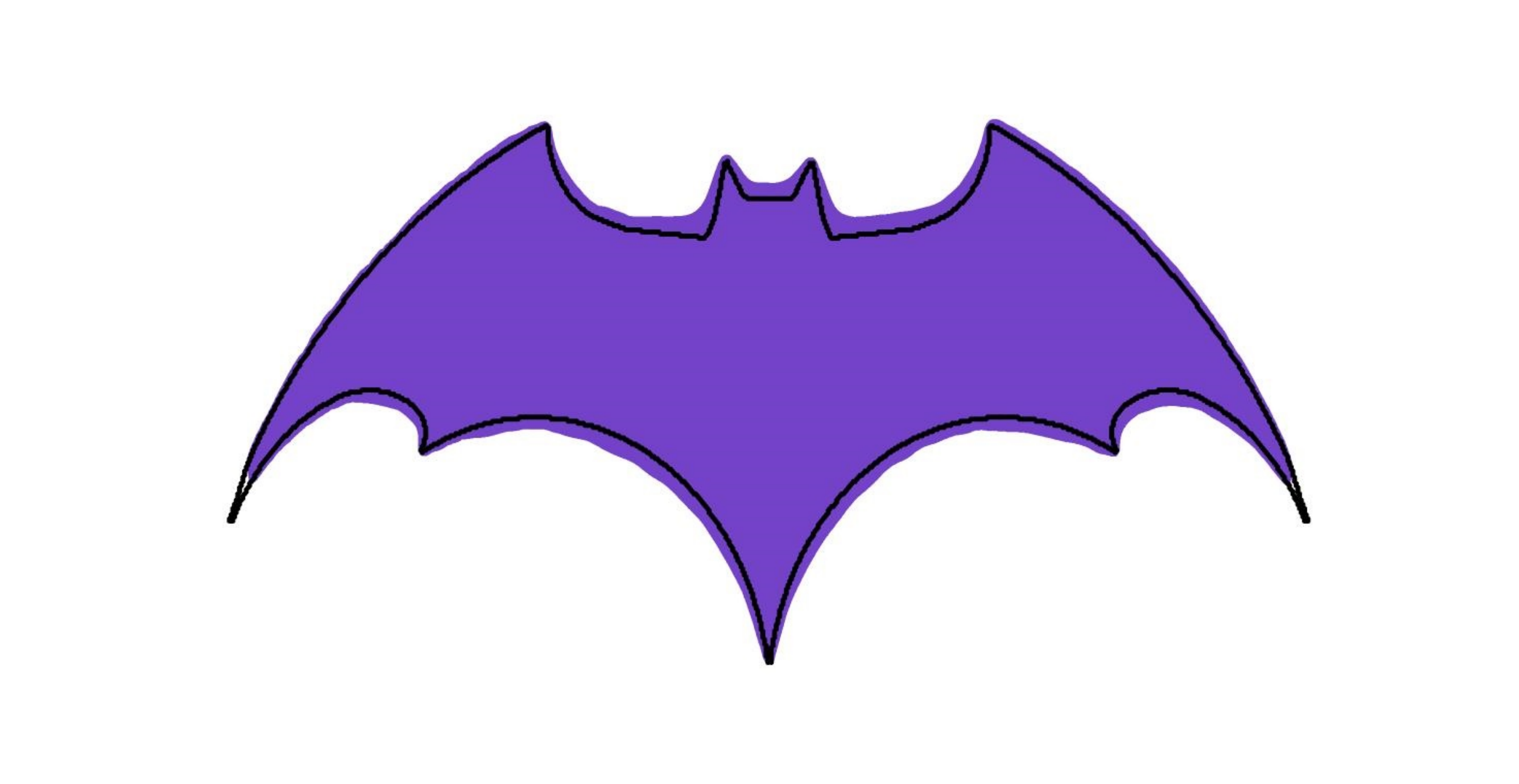}
\caption{A filled Julia set in the shape of Batman.}
\end{center}

\end{figure}

\end{example}

\begin{example}
{\em The KLMY Julia set.}

The following is an example of a disconnected Julia set representing the initials KLMY.

Figure 5 is a representation of the filled Julia set of the polynomial $\tilde{P}_n$
with $n=2000$ and $s=1/2000$. The Julia set was plotted using a resolution of $5000 \times 5000$, and the computations took approximately $340$ seconds.

\begin{figure}[h!t!b]
\begin{center}

\includegraphics[width=.8\textwidth]{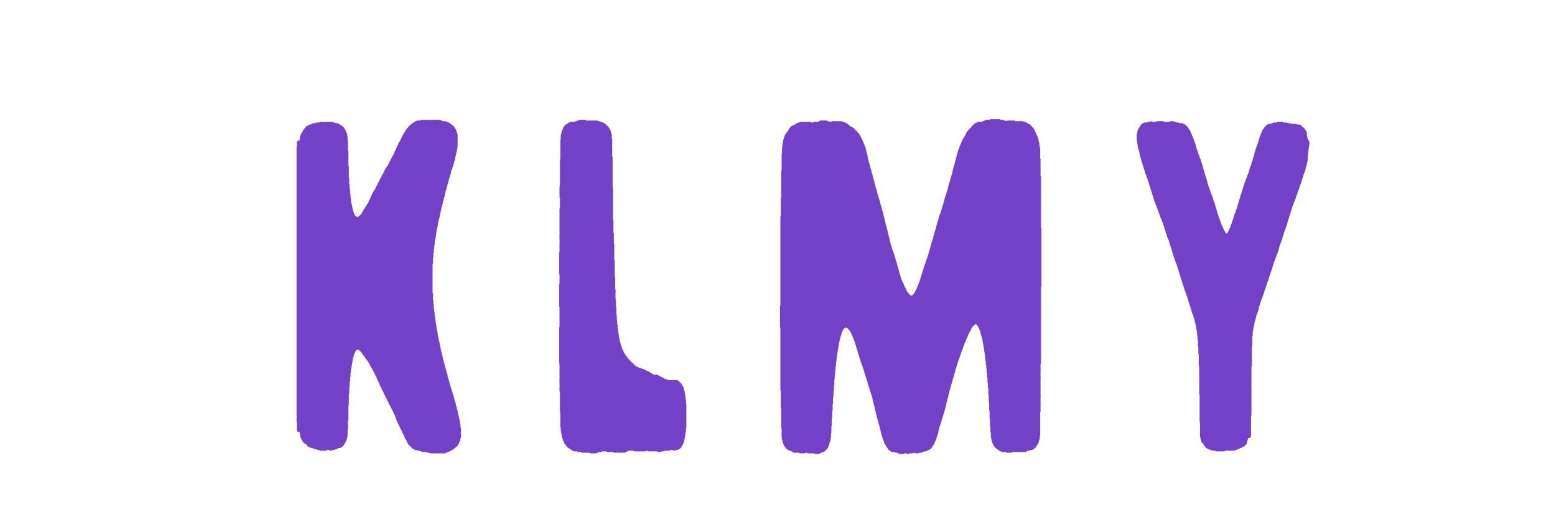}
\caption{A filled Julia set in the shape of the initials KLMY.}
\end{center}

\end{figure}

\end{example}

We note that the filled Julia set has infinitely many connected components, although only four of them are visible. This is most likely due to a lack of precision of the algorithm in \textsc{matlab}. In order to be able to see the smaller components, one can instead use the software \textsc{ultrafractal} to plot the Julia set. This is what we did in the following examples.

\begin{example}
{\em The fish-heart-diamond Julia set.}

The following is an example of a disconnected Julia set representing the shapes of a fish, a heart and a diamond.

Figure 6 is a representation of the filled Julia set of the polynomial $\tilde{P}_n$
with $n=550$ and $s=1/550$, obtained with \textsc{ultrafractal}.

\begin{figure}[h!t!b]
\begin{center}

\includegraphics[width=.7\textwidth]{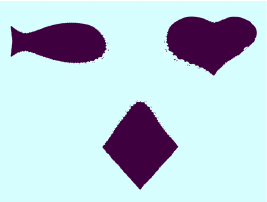}
\caption{A filled Julia set in the shape of a fish, a heart and a diamond.}
\end{center}

\end{figure}

Figure 7 is a zoomed portion of the boundary of the fish where one can see small distorted copies of the heart and the diamond.

\begin{figure}[h!t!b]
\begin{center}

\includegraphics[width=.7\textwidth]{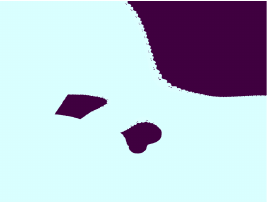}
\caption{Zoomed portion of the boundary.}
\end{center}

\end{figure}

\end{example}

\acknowledgments{The authors thank Peter Lin for helping with the proof of Theorem \ref{t:connected} and Thomas Ransford for the proof of Lemma \ref{lemmaTom}. The authors also thank Amie Wilkinson and Lasse Rempe-Gillen for helpful discussions, as well as Siyuan Li, Xiao Li and Ryan Pachauri for helping with the numerical examples.}

\bibliographystyle{amsplain}

\end{document}